\documentclass[12pt,reqno]{amsart}
\setbox0=\hbox{$+$}
\newdimen\plusheight
\plusheight=\ht0
\def\+{\;\lower\plusheight\hbox{$+$}\;}

\setbox0=\hbox{$-$}
\newdimen\minusheight
\minusheight=\ht0
\def\-{\;\lower\minusheight\hbox{$-$}\;}

\setbox0=\hbox{$\cdots$}
\newdimen\cdotsheight
\cdotsheight=\plusheight
\def\cds{\lower\cdotsheight\hbox{$\cdots$}}

\renewcommand{\(}{\left\(}
\renewcommand{\)}{\right\)}
\renewcommand{\[}{\left[}
\renewcommand{\]}{\right]}

\def\NI{\noindent}

\numberwithin{equation}{section}
 \theoremstyle{plain}
\newtheorem{theorem}{Theorem}[section]
\newtheorem{lemma}[theorem]{Lemma}

\begin{document}
\centerline{\bf \large Some Congruences of a Restricted
Bipartition Function}\vskip
 5mm

 \centerline{\bf \small Nipen Saikia and Chayanika Boruah}\vskip 3mm

\begin{center}{\footnotesize Department of Mathematics, Rajiv Gandhi University, \\Rono
Hills, Doimukh-791112, Arunachal Pradesh, India.\\
E. Mail: nipennak@yahoo.com; cboruah123@gmail.com}\end{center}\vskip 3mm

 \NI{\bf Abstract:} Let $c_N(n)$ denotes the number of bipartitions $(\lambda, \mu)$ of  a positive integer $n$ subject to
 the restriction that each part of $\mu$ is divisible by $N$. In this paper, we prove some congruence properties of the function $c_N(n)$ for $N=7$,
 11, and $5l$, for any integer $l\ge 1$, by employing Ramanujan's theta-function identities.\vskip 3mm

 \NI {\bf Keywords and Phrases:} Partition congruence; Restricted bipartition;
 Ramanujan's theta-functions.\vskip
 3mm

 \NI{\bf Mathematics Subject Classifications:} 05A17; 11P83.\vskip
 3mm

\section{Introduction}
A bipartition of a positive integer $n$ is an ordered pair of partitions $(\lambda, \mu)$  such that the sum of all
of the parts equals $n$. If $c_N(n)$ counts the number of bipartitions $(\lambda,\mu)$ of $n$ subject to the restriction that each part of 
$\mu$ is divisible by $N$, then the generating function of $c_N(n)$ \cite{sinick} is given by 
\begin{equation}\label{ps}
\sum_{n=0}^{\infty} c_N(n)q^n=\frac{1}{(q;q)_{\infty}(q^N;q^N)_{\infty}},
\end{equation}where 
\begin{equation}
(a;q)_{\infty}=\prod_{n=0}^{\infty}(1-aq^n).
\end{equation}
The partition function $c_N(n)$ is first studied by Chan \cite{chan1} for the particular case $N=2$ by considering the function $a(n)$ defined by
 \begin{equation}
\sum_{n=0}^{\infty}a(n)q^n=\frac{1}{(q;q)_{\infty}(q^2;q^2)_{\infty}}
\end{equation} Chan \cite{chan1} proved that, for $n\geq 0$
\begin{equation}\label{qs*}a(3n+2)\equiv 0~(mod~3).\end{equation}
Kim\cite{kim} gave a  combinatorial interpretation \eqref{qs*}. In next paper, Chan\cite{chan2} showed that, for $k\geq1$ and $n\geq0$
\begin{equation}a(3^k n+s_k)\equiv0 ~(mod ~3^{k+\delta(k)}),\end{equation}
where $s_k$ is the reciprocal modulo $3^k$ of 8 and $\delta(k)=1$ if $k$ is even, and 0 otherwise.
Inspired by the work of Ramanujan on the standard partition function $p(n)$, Chan\cite{chan2} asked whether there are any other congruences of the following form
$a(ln+k)\equiv~0~(mod~l),$
where $l$ is prime and $0\leq k\leq l$. Sinick\cite{sinick} answered Chan's question in the negative by considering restricted bipartition function $c_N(n)$ defined in \eqref{ps}.
Wang and Liu \cite{wang} established several infinite families of congruences for $c_5(n)$ modulo 3. For example, they proved that
\begin{equation}c_5\left(3^{2\alpha+1}n+\frac{7\cdot3^{2\alpha}+1}{4}\right)\equiv0~(mod~3), ~\alpha\geq1 ,~n\geq0.\end{equation}

Baruah and Ojha \cite{ndb} also proved some congruences for some particular cases of $C_N(n)$ by considering the generalised partition function $p_{[c^l d^m]}(n)$ defined by
\begin{equation}\sum_{n=0}^{\infty}p_{[c^l d^m]}(n)q^n=\frac{1}{(q^c;q^c)_{\infty}^l(q^d;q^d)_{\infty}^{m}},\end{equation}
and using Ramanujan's modular equations. Clearly,
$c_N(n)=p_{[1^1 N^1]}(n)$. For example, Baruah and Ojah \cite{ndb} proved that
\begin{equation} p_{[1^1 3^1]}(4n+j)\equiv 0~(mod~2), \mbox{for} ~ j=2, 3 \end{equation}and
\begin{equation} p_{[1^1 7^1]}(8n+7)\equiv 0~(mod~2).\end{equation}
Ahmed et al. \cite{ndbet} investigated the function $C_N(n)$ for $N=$ 3 and 4 and proved some congruences modulo 5. They also gave alternate proof of some congruences due to Chan \cite{chan1}.

In this paper, we investigate the restricted bipartition function $c_N(n)$ for  $n=$7, 11, and 5$l$, for any integer $l\ge 1$, and prove some congruences modulo 2, 3 and 5 by using Ramanujan's theta-function identities.
In Section 3, we prove congruences modulo 2 for $c_7(n)$. For example, we prove,  for $\alpha\geq0 $
\begin{equation}c_7\left(2^{2\alpha+1}n+\frac{5\cdot2^{2\alpha}+1}{3}\right)\equiv0~(mod~2).\end{equation}
In Section 4, we deal with the function $c_{11}(n)$ and establish that, if $p$ is a prime,  $1\leq j\leq p-1,$ and $\alpha\ge 0$, then
\begin{equation}c_{11}\left(4p^{2\alpha+1}(pn+j)+\frac{p^{2\alpha+2}+1}{2}\right)\equiv0~(mod~2).\end{equation}
In Section 5, we show that, for any integer $l\ge 1$, $c_{5l}(5n+4)\equiv 0 (mod~5)$. We also prove congruences modulo 3 for $c_{15}(n).$
Section 2 is devotd to record some preliminary results.

\section{Preliminary Results}
Ramanujan's general theta function $f(a, b)$ is defined by
\begin{equation}\label{fab}f(a, b)=\sum_{n=0}^\infty a^{n(n+1)/2}b^{n(n-1)/2},~  |ab|<1.
\end{equation}Three important special cases of $f(a, b)$ are
\begin{equation}\label{phid}\phi(q):=f(q, q)=\sum_{n=-\infty}^{\infty}q^{n^2}=\frac{(q^2;q^2)^5_{\infty}}{(q;q)^2_{\infty}(q^4;q^4)^2_{\infty}},
\end{equation}
\begin{equation}\label{psid}\psi(q):=f(q, q^3)=\sum_{n=0}^{\infty}q^{n(n+1)/2}=\frac{(q^2;q^2)^2_{\infty}}{(q;q)_{\infty}},
\end{equation}
\begin{equation}\label{fd}f(-q):=f(-q, -q^2)=\sum_{n=-\infty}^{\infty}(-1)^nq^{n(3n+1)/2}=(q;q)_{\infty},
\end{equation}
Ramanujan also defined the function $\chi(q)$ as
\begin{equation}\label{chid}\chi(q)=(-q;q^2)_\infty.
\end{equation}

\begin{lemma}\label{qprop} For any prime $p$ and positive integer $m$, we have
$$(q^{pm};q^{pm})_\infty\equiv (q^m;q^m)_\infty^p~(mod~p).$$\end{lemma}
\begin{proof}Follows easily from binomial theorem.
\end{proof}

\begin{lemma}\label{f1}\cite[p. 315]{berndt} We have
 \begin{equation}\psi(q)\psi(q^7)=\phi(q^{28})\psi(q^8)+q\psi(q^{14})\psi(q^2) +q^6\psi(q^{56})\phi(q^4).\end{equation}
\end{lemma}

\begin{lemma}\label{f2}We have \begin{equation}\psi(q)\psi(q^7)\equiv (q;q)_{\infty}^{3}(q^7;q^7)_{\infty}^{3}~(mod~2).\end{equation}
\end{lemma}
\begin{proof}
From \eqref{psid}, we have
 \begin{equation}\label{ps71}\psi(q)\psi(q^7)=\frac{(q^2;q^2)^2_{\infty}(q^{14},q^{14})^2_{\infty}}{(q;q)_{\infty}(q^7;q^7)_{\infty}}.\end{equation}
 Simplifying \eqref{ps71} using Lemma \ref{qprop} with $p=2$, we  arrive at the desired result.
\end{proof}

\begin{lemma}\cite[p. 286, Eqn.(3.19)]{ndb2} We have
\begin{equation}\label{phiminus}
\phi(-q)=\frac{(q;q)^2_{\infty}}{(q^2;q^2)_{\infty}}.\end{equation}
\begin{equation}\label{psiminus}\psi(-q)=\frac{(q;q)_{\infty}(q^4;q^4)_{\infty}}{(q^2;q^2)_{\infty}}.\end{equation}
\begin{equation}f(q)=\frac{(q^2;q^2)^3_{\infty}}{(q;q)_{\infty}(q^4;q^4)_{\infty}}.\end{equation}
\begin{equation}\label{chiplus}\chi(q)=\frac{(q^2;q^2)^2_{\infty}}{(q^2;q^2)_{\infty}(q^4;q^4)_{\infty}}.
\end{equation}
\end{lemma}

\begin{lemma}\label{f3} \cite[p. 372]{berndt1} We have
\begin{equation}\label{qs56} \psi(q)\psi(q^{11})=\phi(q^{66})\psi(q^{12})+qf(q^{44},q^{88})f(q^2,q^{10})+q^{22}f(q^{22},q^{110})f(q^8,q^4)$$$$\hspace{-5cm}+q^{15}\psi(q^{132})\phi(q^6).\end{equation}
\end{lemma}

\begin{lemma}\label{f6}\cite[p. 350, Eqn.(2.3)]{berndt}We have
\begin{equation}f(q,q^2)=\frac{\phi(-q^3)}{\chi(-q)},\end{equation}\end{lemma}

where \begin{equation}\label{chi}\chi(-q)=(q;q)_{\infty}/(q^2;q^2)_{\infty}\end{equation}
\begin{lemma}\label{mo} We have
\begin{equation}f(q^{11};q^{22})\equiv (q^{11};q^{11})_{\infty}~(mod~2).\end{equation}
\end{lemma}
\begin{proof}Employing \eqref{phiminus} in Lemma \ref{f6} and simplifying using Lemma \ref{qprop} with $p=2$, we obtain
\begin{equation}\label{cb}f(q;q^2)\equiv(q;q)_{\infty}~(mod~2).\end{equation}
Replacing $q$ by $q^{11}$ in \eqref{cb}, we arrive at  the desired result.
\end{proof}

\begin{lemma}\label{f10}\cite[p. 51, Example (v)]{berndt}We have
\begin{equation}f(q,q^5)=\psi(-q^3)\chi(q)
\end{equation}
\end{lemma}

\begin{lemma}\label{f29}We have
\begin{equation}f(q,q^5)\equiv \frac{(q^3;q^3)^3_{\infty}}{(q;q)_{\infty}}~(mod~2)\end{equation}
\end{lemma}

\begin{proof}Employing \eqref{psiminus} and \eqref{chiplus} in Lemma \ref{f10}, we obtain
\begin{equation}\label{f15m1}f(q,q^5)=\frac{(q^3;q^3)_\infty(q^{12};q^{12})_\infty(q^2;q^2)_\infty^2}{(q^6;q^6)_\infty(q;q)_\infty(q^4;q^4)_\infty}.
\end{equation}Simplifying \eqref{f15m1} using Lemma \ref{qprop} with $p=2$, we complete the proof.
\end{proof}

\begin{lemma}\label{f5}\cite[p. 5, Eqn.(2.5) ]{hirs} We have
\begin{equation}\frac{(q^3;q^3)^3_{\infty}}{(q;q)_{\infty}}=\frac{(q^4;q^4)^3_{\infty}(q^6;q^6)^2_{\infty}}
{(q^2;q^2)^2_{\infty}(q^{12};q^{12})_{\infty}}+q\frac{(q^{12};q^{12})^3_{\infty}}{(q^4;q^4)_{\infty}}.\end{equation}

\end{lemma}

\begin{lemma}\label{f8} \cite[Theorem 2.1]{cui} For any odd prime p ,
\begin{equation}
\psi(q)=\sum_{k=0}^{\frac{p-3}{2}}q^{\frac{k^2+k}{2}}f(q^{\frac{p^2+(2k+1)p}{2}},q^{\frac{p^2-(2k+1)p}{2}})+q^{\frac{p^2-1}{8}}\psi(q^{p^2}),
\end{equation}
where
$$\frac{k^2+k}{2}\not\equiv\frac{p^2-1}{8}~(mod~p)~ ~for, ~ 0\leq k\leq\frac{p-3}{2}.$$
\end{lemma}

\begin{lemma}\label{f21} \cite[Theorem 2.2]{cui} For any prime $p\geq5$, we have
\begin{equation}f(-q)=\sum_{\substack {k=\frac{-p-1}{2}\\{k\ne\frac{\pm p-1}{6}}}}^{\frac{p-1}{2}}(-1)^kq^{\dfrac{3k^2+k}{2}}f\left(-q^{\frac{3p^2+(6k+1)p}{2}},-q^{\frac{3p^2-(6k+1)p}{2}}\right)+(-1)^{\frac{\pm p-1}{6}} q^{\frac{p^2-1}{24}}f(-q^{p^2}),\end{equation}
where
$\dfrac{\pm p-1}{6}:=
\begin{cases}
\dfrac{p-1}{6} , & if~ p\equiv 1~(mod~6),\\
\dfrac{-p-1}{6} , & if~ p\equiv -1~(mod~6).
\end{cases}
$\end{lemma}

\begin{lemma}\label{f11} \cite{hirs1} We have
$$\frac{1}{(q;q)_{\infty}}=\frac{(q^{25};q^{25})^6_{\infty}}{(q^5;q^5)^6_{\infty}}(F^4(q^5)+qF^3(q^5)+2q^2F^2(q^5)+3q^3F(q^5)+5q^4-3q^5F^{-1}(q^5)$$
\begin{equation} +2q^6F^{-2}(q^5)-q^7F^{-3}(q^5)+q^8F^{-4}(q^5)),\end{equation}
where $F(q):=q^{-1/5}R(q)$ and $R(q)$ is Roger-Ramanujan continued fraction defined by
$$R(q):=\frac{q^{1/5}}{1}_{+}\frac{q}{1}_{+}\frac{q^2}{1}_{+}\frac{q^3}{1}_{+
\cdots}, \hspace{.4cm}\vert q\vert<1.$$
\end{lemma}

\begin{lemma}\label{f12} \cite[p.345, Entry 1(iv)]{berndt} We have
\begin{equation}(q;q)^3_{\infty}=(q^9;q^9)^3_{\infty}\left(4q^3W^2(q^3)-3q+W^{-1}(q^3)\right),\end{equation}
where $W(q)=q^{-1/3}G(q)$
and $G(q)$ is the Ramanujan's cubic continued fraction defined by
$$G(q):=\frac{q^{1/3}}{1}_{+}\frac{q+q^2}{1}_{+}\frac{q^2+q^4}{1}_{+
\cdots}, \hspace{.4cm}\vert q\vert<1.$$
\end{lemma}

\section{Congruence Identities for $c_7(n)$}

\begin{theorem}\label{qs1*} We have
\begin{equation}\sum_{n=0}^{\infty}c_7(2n+1)q^n\equiv(q;q)_{\infty}(q^7;q^7)_{\infty}~(mod~2).\end{equation}
\end{theorem}

\begin{proof}For $N=7$ in \eqref{ps}, we have
\begin{equation}\label{qs51}\sum_{n=0}^{\infty}c_7(n)q^n=\frac{1}{(q;q)_{\infty}(q^7;q^7)_{\infty}}.\end{equation}
Employing \eqref{ps71} in \eqref{qs51}, we obtain \begin{equation}\label{qs571}\sum_{n=0}^{\infty}c_7(n)q^n=\frac{\psi(q)\psi(q^7)}{(q^2;q^2)_{\infty}^{2}(q^{14};q^{14})^{2}_{\infty}}.\end{equation}
Employing Lemma \ref{f1} in \eqref{qs571}, we obtain
\begin{equation}\label{qs} \sum_{n=0}^{\infty}c_7(n)q^n=\frac{1}{(q^2;q^2)^2_{\infty}(q^{14};q^{14})_{\infty}^2}\[\phi(q^{28})\psi(q^8)+q\psi(q^{14})\psi(q^{2})+q^6\psi(q^{56})\phi(q^4)\].  \end{equation}
Extracting the terms involving $q^{2n+1}$, dividing by $q$  and replacing $q^2$ by $q$ in \eqref{qs} , we get
\begin{equation}\label{qs57}\sum_{n=0}^{\infty}c_7(2n+1)q^n=\frac{1}{(q;q)_{\infty}^{2}(q^7;q^7)_{\infty}^{2}}[\psi(q^7)\psi(q)].\end{equation}
Employing Lemma \ref{f2} in \eqref{qs57}, we complete the proof.
\end{proof}

\begin{theorem}\label{th7} We have
$$(i)~\sum_{n=0}^{\infty}c_7(4n+3)q^n\equiv (q^2;q^2)_{\infty}(q^{14};q^{14})_{\infty}~(mod~2).$$
$$\hspace{-4.5cm}(ii)~c_7(8n+7)\equiv 0~(mod~2).$$
\end{theorem}

\begin{proof}

From Lemma \ref{qs1*}, we obtain
\begin{equation}\label{qs2}\sum_{n=0}^{\infty}c_7(2n+1)q^n\equiv \frac{(q^7;q^7)_{\infty}^3(q;q)^3_\infty}{(q^7;q^7)_{\infty}^2(q;q)^2_{\infty}}~(mod~2).\end{equation}
Employing Lemma \ref{f2} in \eqref{qs2}, we obtain
\begin{equation}\label{qs3}\sum_{n=0}^{\infty}c_7(2n+1)q^n \equiv\frac{\psi(q)\psi(q^7)}{(q^2;q^2)_{\infty}(q^{14};q^{14})_{\infty}}~(mod~2).\end{equation}
Emplyoing Lemma \ref{f1} in \eqref{qs3}, extracting the terms involving $q^{2n+1},$ dividing by $q$ and replacing $q^2$ by $q,$ we obtain
\begin{equation}\label{qs4}\sum_{n=0}^{\infty}c_7(4n+3)q^n\equiv\frac{1}{(q;q)_{\infty}(q^7;q^7)_{\infty}}\psi(q)\psi(q^7)~(mod~2). \end{equation}
Employing  Lemma \ref{f2} in \eqref{qs4} and simplifying using Lemma \ref{qprop} with $p=2$, we arrive at (i).

All the terms on the right hand  side of (i) are of the form $q^{2n}$. Extracting the terms involving $q^{2n+1}$ on both sides of (i), we complete the proof of (ii).
\end{proof}

\begin{theorem}
For all $ n \geq0$, we have
$$\hspace{-10cm}(i)~ c_7(14n+7)\equiv 0~(mod~2),$$
$$\hspace{-10cm}(ii)~c_7(14n+9)\equiv 0~(mod~2),$$
$$\hspace{-10cm}(iii)~ c_7(14n+13)\equiv 0~(mod~2).$$
\end{theorem}

\begin{proof}Employing \eqref{fd} in Lemma \ref{qs1*}, we obtain
\begin{equation}\label{qs6}\sum_{n=0}^{\infty}c_7(2n+1)q^n\equiv(q^7;q^7)_{\infty}\sum_{n=0}^{\infty}(-1)^nq^{n(3n+1)/2}~(mod~2).\end{equation}
Extracting those terms on each side of \eqref{qs6} whose power of $q$ is of the form $7n+3$, $7n+4$, and $7n+6$ and employing the fact that  there exists no integer $n$ such that $n(3n+1)/2$ is congruent to 3, 4, and 6 modulo 7, we obtain
\begin{equation}\label{oo1}\sum_{n=0}^{\infty}c_7(14n+7)q^{7n+3}\equiv\sum_{n=0}^{\infty}c_7(14n+9)q^{7n+4}\equiv\sum_{n=0}^{\infty}c_7(14n+13)q^{7n+6}\equiv 0~(mod~2).\end{equation}
Now (i), (ii), and (iii) are obvious from \eqref{oo1}.\end{proof}

\begin{theorem} For $\alpha\geq 1 $, we have
\begin{equation}\label{qs52}\sum_{n=0}^{\infty}c_7\left(2^{2\alpha+1}n+\frac{2^{2\alpha+1}+1}{3}\right)q^n\equiv(q;q)_{\infty}(q^7;q^7)_{\infty}~(mod~2).\end{equation}
\end{theorem}
\begin{proof}
We proceed by induction on $\alpha.$
Extracting the terms involving $q^{2n}$  and replacing $q^2$ by $q$ in Theorem \ref{th7}(i), we obtain
\begin{equation}\sum_{n=0}^{\infty}c_7(8n+3)q^n\equiv(q;q)_{\infty}(q^7;q^7)_{\infty}~(mod~2),\end{equation}
which corresponds to the case $\alpha=1$.
Assume, that the result is true for $\alpha=k\ge 1$, so that
\begin{equation}\label{o2}\sum_{n=0}^{\infty}c_7\left(2^{2k+1}n+\frac{2^{2k+1}+1}{3}\right)q^n\equiv(q;q)_{\infty}(q^7;q^7)_{\infty}~(mod~2).\end{equation}
Employing Lemma \ref{f2} in \eqref{o2}, we obtain
\begin{equation}\label{qs7}\sum_{n=0}^{\infty}c_7\left(2^{2k+1}n+\frac{2^{2k+1}+1}{3}\right)q^n \equiv \frac{\psi(q)\psi(q^7)}{(q;q)^2_{\infty}(q^7;q^7)_{\infty}^{2}}~(mod~2).\end{equation}
Employing Lemma \ref{f1} in \eqref{qs7} and extracting the terms involving $q^{2n+1}$, dividing by $q$ and replacing $q^2$ by $q$, we obtain
\begin{equation}\label{qs8}\sum_{n=0}^{\infty}c_7\left(2^{ 2k+1}(2n+1)+\frac{2^{2k+1}+1}{3}\right)q^n\equiv \frac{\psi(q)\psi(q^7)}{(q;q)_{\infty}(q^7;q^7)_{\infty}}~(mod~2).\end{equation}
Simplifying \eqref{qs8} using Lemma \ref{f2} and Lemma \ref{qprop} with $p=2$, we obtain
\begin{equation}\label{qs9}\sum_{n=0}^{\infty}c_7\left(2^{2(k+1)}n+\frac{2^{2(k+1)+1}+1}{3}\right)q^n\equiv(q^2;q^2)_{\infty}(q^{14};q^{14})_{\infty}~(mod~2).\end{equation}
Extracting the terms involving $q^{2n}$ and replacing $q^2$ by $q$ in \eqref{qs9},we obtain
\begin{equation}\sum_{n=0}^{\infty}c_7\left(2^{2(k+1)+1}n+\frac{2^{2(k+1)+1}+1}{3}\right)q^n\equiv(q;q)_{\infty}(q^7;q^7)_{\infty}~(mod~2).\end{equation}
which is the $\alpha=k+1$ case. Hence, the proof is complete.
\end{proof}

\begin{theorem} For $\alpha\geq0$, we have
\begin{equation}c_7\left(2^{2\alpha+1}n+\frac{5\cdot2^{2\alpha}+1}{3}\right)\equiv0~(mod~2).\end{equation}
\end{theorem}
\begin{proof} All the terms in the right hand side of \eqref{qs9}, are of the form $q^{2n}$, so extracting the coefficients of $q^{2n+1}$
 on both sides of \eqref{qs9} and replacing $k$ by $\alpha$, we obtain
\begin{equation}\label{qs80}c_7\left(2^{2(\alpha+1)+1}n+\frac{5\cdot2^{2(\alpha+1)}+1}{3}\right)\equiv0~(mod~2).
\end{equation}
Replacing $\alpha+1$ by $\alpha$ in \eqref{qs80}, completes the proof.
\end{proof}

\begin{theorem} If any prime $p\geq5 $, $\left(\dfrac{-7}{p}\right)=-1$, and $\alpha\geq0$, then
\begin{equation}\label{qs83}c_7\left(2^{2\alpha+1}p^2n+\frac{2^{2\alpha+1}p(3j+p)+1}{3}\right)\equiv0~(mod~2). \end{equation}
\end{theorem}
\begin{proof} Employing Lemma \ref{f21} in \eqref{qs52}, we obtain
$$\hspace{-10cm}\sum_{n=0}^{\infty}c_7\left(2^{2\alpha+1}n+\frac{2^{2\alpha+1}+1}{3}\right)q^n$$
$$\equiv\left(\sum_{\substack {k=\frac{-p-1}{2}\\{k\ne\frac{\pm p-1}{6}}}}^{\frac{p-1}{2}}(-1)^kq^{\frac{3k^2+k}{2}}f\left(-q^{\frac{3p^2+(6k+1)p}{2}},-q^{\frac{3p^2-(6k+1)p}{2}}\right)+(-1)^{\frac{\pm p-1}{6}} q^{\frac{p^2-1}{24}}f(-q^{p^2})\right)$$
\begin{equation}\label{qs87}\times\left(\sum_{\substack {k=\frac{-p-1}{2}\\{k\ne\frac{\pm p-1}{6}}}}^{\frac{p-1}{2}}(-1)^mq^{7\cdot\frac{3m^2+m}{2}}f\left(-q^{7\cdot\frac{3p^2+(6m+1)p}{2}},-q^{7\cdot\frac{3p^2-(6m+1)p}{2}}\right)
+(-1)^{\frac{\pm p-1}{6}} q^{7\cdot\frac{p^2-1}{24}}f(-q^{7p^2})\right)~(mod~2).
\end{equation}
We consider the congruence
\begin{equation}\label{cou}\frac{3k^2+k}{2}+7\cdot\frac{3m^2+m}{2}\equiv \frac{8p^2-8}{24}~(mod~p),\end{equation}
where $-(p-1)/2\leq k, m\leq (p-1)/2$. The congruence \eqref{cou}  is equivalent to
\begin{equation}\label{op1}(6k+1)^2+7(6m+1)^2\equiv 0~(mod~p)\end{equation} and for $(\frac{-7}{p})=-1,$  the congruence \eqref{op1} has uniqe solution $ k=m =\frac{\pm p-1}{6}.$
Extracting terms containing $q^{pn+\frac{p^2-1}{3}}$ from both sides of \eqref{qs87} and replacing $q^p$ by $q$ ,  we obtain
\begin{equation}\label{qs88}\sum_{n=0}^{\infty}c_7\left(2^{2\alpha+1}pn+\frac{2^{2\alpha+1}p^2+1}{3}\right)q^n\equiv (q^p;q^p)_{\infty}(q^{7p};q^{7p})~(mod~2).\end{equation}
Extracting the coefficients of $q^{pn+j}$,  for $1\leq j \leq p-1,$ on both sides of \eqref{qs88} and simplifying, we arrive at the desired  result.
\end{proof}

\section{Congruence Identities for $c_{11}(n)$}

\begin{theorem}\label{th5} We have
$$ \sum_{n=0}^{\infty}c_{11}(4n+1)q^n\equiv\frac{(q^2;q^2)_{\infty}^2}{(q;q)_{\infty}}= \psi(q)~(mod~2).$$
\end{theorem}
\begin{proof}
Setting $N=11$ in \eqref{ps}, we obtain
\begin{equation}\label{qs58}\sum_{n=0}^{\infty}c_{11}(n)q^n=\frac{1}{(q;q)_{\infty}(q^{11};q^{11})_{\infty}}.\end{equation}
Employing \eqref{psid} in \eqref{qs58}, we obtain \begin{equation}\label{qs11}\sum_{n=0}^{\infty}c_{11}(n)q^n=\frac{\psi(q)\psi(q^{11})}{(q^2;q^2)_{\infty}^{2}(q^{22};q^{22})^{2}_{\infty}}.\end{equation}
Employing Lemma \ref{f3} in \eqref{qs11}, extracting the terms involving $q^{2n+1}$, dividing by $q$, and replacing $q^2$ by $q$, we obtain
\begin{equation}\label{qs12}\sum_{n=0}^{\infty}c_{11}(2n+1)q^n=\frac{1}{(q;q)^2_{\infty}(q^{11};q^{11})^2_{\infty}}[f(q^{22},q^{44})f(q,q^5)+q^7\psi(q^{66})\phi(q^3)].\end{equation}
Employing Lemmas \ref{f29} and \ref{f5} in \eqref{qs12}, we find that
$$\hspace{-8cm}\sum_{n=0}^{\infty}c_{11}(2n+1)q^n\equiv\frac{1}{(q^2;q^2)_{\infty}(q^{22};q^{22})_{\infty}}$$
\begin{equation}\label{qs13}\times\[f(q^{22},q^{44})_{\infty}\left(\frac{(q^4;q^4)^3_{\infty}(q^6;q^6)^2_{\infty}}{(q^2;q^2)^2_{\infty}(q^{12};q^{12})_{\infty}}+q\frac{(q^{12};q^{12})^3}{(q^4;q^4)_{\infty}}\right)+q^7\psi(q^{66})\phi(q^3)\]~(mod~2).\end{equation}
Extracting the terms involving $q^{2n}$ and replacing $q^2$ by $q$ on both sides of  \eqref{qs13} and simplifying using Lemma \ref{qprop} with
$p=2$, we obtain
\begin{equation}\label{qs17}\sum_{n=0}^{\infty}c_{11}(4n+1)q^n\equiv\frac{1}{(q;q)_{\infty}(q^{11};q^{11})_{\infty}}f(q^{11};q^{22})(q^2;q^2)^2_{\infty}~(mod~2).\end{equation}
Employing  Lemma \ref{mo} in \eqref{qs17} and using \eqref{psid}, we complete the proof.
\end{proof}

\begin{theorem} For any prime $p$ and any integer $\alpha\ge 0$, we have
$$\label{qs19} \sum_{n=0}^{\infty}c_{11}\left(4p^{2\alpha}n+\frac{p^{2\alpha}+1}{2}\right)q^n \equiv \psi(q)~(mod~2).$$
\end{theorem}
\begin{proof}We proceed by induction on $\alpha$. The case $\alpha=0$ corresponds to the congruence Theorem \ref{th5}. Suppose that the theorem  holds for  $\alpha=k\ge 0$, so that
\begin{equation}\label{qs19a} \sum_{n=0}^{\infty}c_{11}\left(4p^{2k}n+\frac{p^{2k}+1}{2}\right)q^n \equiv \psi(q)~(mod~2).\end{equation}
Employing Lemma \ref{f8} in \eqref{qs19a}, extracting the terms involving $q^{pn+\frac{p^2-1}{8}}$ on both sides of \eqref{qs19a}, dividing by $q^{\frac{p^2-1}{8}}$ and  replacing $q^p$ by $q$, we obtain
\begin{equation}\label{qs20}\sum_{n=0}^{\infty}c_{11}\left(4p^{2k +1}n+\frac{p^{2(k +1)}+1}{2}\right)q^n \equiv \psi(q^p)~(mod~2).\end{equation}
Extracting the terms containing $q^{pn}$ from both sides of \eqref{qs20} and replacing $q^p$ by $q$, we arrive at
\begin{equation}\label{qs21}\sum_{n=0}^{\infty}c_{11}\left(4p^{2(k+1)}n+\frac{p^{2(k+1)}+1}{2}\right)q^n \equiv \psi(q)~(mod~2).\end{equation}
which shows that the theorem is true for $\alpha=k+1$. Hence, the proof is complete.
\end{proof}

\begin{theorem}
 For any prime $p$ and integers $\alpha\ge 0$ and $1\leq j\leq p-1$, we have
\begin{equation}\label{qs22}c_{11}\left(4p^{2\alpha+1}(pn+j)+\frac{p^{2\alpha+2}+1}{2}\right)\equiv0~(mod~2).\end{equation}
\end{theorem}
\begin{proof}
Extracting the coefficients of $q^{pn+j}$, for $1\leq j\leq p-1$ on both sides of \eqref{qs20} and replacing $k$ by $\alpha$, we arrive at the desired result.
\end{proof}

\section{Congruence Identities for $c_{5l}(n)$}

\begin{theorem}For any positive integer $l$, we have
\begin{equation}c_{5l}(5n+4)\equiv 0~(mod~5).\end{equation}
\end{theorem}

\begin{proof}Setting $N=5l$ in \eqref{ps}, we obtain
\begin{equation}\label{qs95}\sum_{n=0}^{\infty}c_{5l}(n)q^n=\frac{1}{(q;q)_{\infty}(q^{5l};q^{5l})_{\infty}}.\end{equation}
Using Lemma \ref{f11} in \eqref{qs95} and extracting the terms
involving $q^{5n+4}$, dividing by $q^4$ and replacing $q^5$ by
$q$, we obtain
\begin{equation}\label{qs96}\sum_{n=0}^{\infty}c_{5l}(5n+4)q^n=5 \frac{(q^5;q^5)^6_{\infty}}{(q^{l},q^{l})_{\infty}(q;q)^5_{\infty}}.\end{equation}
The desired result follows easily from \eqref{qs96}.
\end{proof}

\begin{theorem} For all $n\geq 0$, we have
$$\hspace{-10cm}(i)~c_{15}(5n+4)\equiv 0~(mod~5),$$
$$\hspace{-9.9cm}(ii)~c_{15}(15n+9)\equiv 0~(mod~3),$$
$$\hspace{-9.8cm}(iii)~c_{15}(15n+14)\equiv 0~(mod~3).$$
\end{theorem}

\begin{proof}
Setting $N=15$ in \eqref{ps}, we obtain \begin{equation}\label{qs70}\sum_{n=0}^{\infty}c_{15}(n)q^n=\frac{1}{(q;q)_{\infty}(q^{15};q^{15})_{\infty}}.\end{equation}
Employing Lemma \ref{f11} in \eqref{qs70}, extracting terms involving $q^{5n+4}$, dividing by $q^4$ and replacing $q^5$ by $q$, we obtain
\begin{equation}\label{qs72}\sum_{n=0}^{\infty}c_{15}(5n+4)q^n=5\frac{(q^5;q^5)^6_{\infty}}{(q^3;q^3)_{\infty}(q;q)^6_{\infty}}.\end{equation}
Now (i) follows from \eqref{qs72}.

Simplifying \eqref{qs72} by using Lemma \ref{qprop} with $p=3$, we obtain
\begin{equation}\label{qs73}\sum_{n=0}^{\infty}c_{15}(5n+4)q^n\equiv 2\frac{(q^{15};q^{15})^2_{\infty}}{(q;q)^3_{\infty}(q;q)^6_{\infty}}\frac{(q;q)^3_{\infty}}{(q;q)^3_{\infty}}
= 2\frac{(q^{15};q^{15})_{\infty}^2(q;q)^3_{\infty}}{(q^3;q^3)^4_{\infty}}~(mod~3).\end{equation}
Employing Lemma \ref{f12} in \eqref{qs73} and simplifying, we obtain
\begin{equation}\label{qs74}\sum_{n=0}^{\infty}c_{15}(5n+4)q^n\equiv 2\frac{(q^{15};q^{15})^2_{\infty}(q^9;q^9)_{\infty}^3}{(q^3;q^3)_{\infty}^4}\left[q^3W^2(q^3)+W^{-1}(q^3)\right]~(mod~3).\end{equation}
Extracting terms involving $q^{3n+1}$ and $q^{3n+2}$ on both sides of \eqref{qs74}, we arrive at (ii) and (iii), respectively.
\end{proof}

\end{document}